\documentclass[11pt]{article}
\usepackage{amsfonts,amsmath, amssymb,latexsym}
\def\Z{{\mathbb Z}}

\def\SL{{\rm SL}}
\def\GL{{\rm GL}}

\def\Disc{{\rm Disc}}

\def\ndeg{{\rm ndeg}}
\def\deg{{\rm deg}}

\def\R{{\mathbb R}}
\def\F{{\mathbb F}}
\def\FF{{\mathcal F}}

\def\Q{{\mathbb Q}}

\def\Z{{\mathbb Z}}

\def\F{{\mathbb F}}
\def\Q{{\mathbb Q}}
\def\N{{\mathbb N}}

\def\e{{\epsilon}}

\newtheorem{theorem}{Theorem}

\newtheorem{lemma}[theorem]{Lemma}

\newenvironment{proof}{\noindent {\bf Proof:}}{$\Box$ \vspace{2 ex}}

\title{Counting $S_5$-fields with a power saving error term}

\author{Arul Shankar and Jacob Tsimerman}

\begin{document}
\maketitle
\begin{abstract} We show how the Selberg $\Lambda^2$-sieve can be used to obtain power saving error terms
in a wide class of counting problems which are tackled using geometry of numbers. Specifically, we give such an
error term for the counting function of $S_5$-quintic fields.
\end{abstract}
\section{Introduction}
Over the past decade there has emerged a large body of work concerned
with counting arithmetic objects by parameterizing them as $G_{\Z}$
orbits on $V_{\Z}$, where $G$ is some reductive algebraic group, and
$V$ is a representation of $G$ (see \cite{dodqf}, \cite{dodpf},
\cite{Bhyper}, \cite{BG3} \cite{BS} \cite{BS5}, \cite{SW}). In certain
applications, particularly relating to low lying zeroes--see \cite{Y},
it is important to not only obtain the asymptotic count, but also to
obtain a power saving error term. i.e. a formula of the
type $$\#\{\textrm{Objects of interest with height less than X} \}=
cX^a\log^b X + O(X^{a-\delta})$$ for some fixed constant $\delta>0$.

In this note, we show how the Selberg $\Lambda^2$-sieve can be used
very generally to obtain such power savings. In particular, we
demonstrate our claim by obtaining the first known power saving for
quintic fields with galois group $S_5$ and bounded discriminant:
\begin{theorem}\label{main}
Define $N_5^{(i)}(X)$ to be the number of quintic fields with galois group $S_5$ having discriminant bounded in absolute value by $X$ with $i$ complex places.
Then $$N_5^{(i)}(X) = d_i\prod_p(1+p^{-2}-p^{-4}-p^{-5})X + O_{\e}(X^{\frac{399}{400}+\e})$$
where $d_0,d_1,d_2$ are $\frac{1}{240},\frac{1}{24},$ and $\frac{1}{16}$  respectively.
\end{theorem}

 We begin with a general sketch of the argument.

\subsection{Sketch of the argument}
Typically, one finds a fundamental domain $F\subset V_{\R}$ for the
action of $G_{\R}$, and one wants to count integral points inside $F$
of bounded height. However, it is not all points that one wants to
count; one partitions the set $V_{\Z}$ into 2 sets $V_{\Z}^{\deg}$ and
$V_{\Z}^{\ndeg}$ where former sets corresponds to objects which are
`degenerate' in some way, and it is only the points in $V_{\Z}^{\ndeg}$
that need to be counted. For example, in the quintic case the
degenerate points correspond to quintic rings $R$ such that
$R\otimes_{\Z}\Q$ is not a quintic field with galois group $S_5$. $F$
is typically not compact and has `cusps' which contain primarily
degenerate points; the method which one uses to estimate the number of
nondegenerate points in the cusp typically yields a power
saving. Letting $F_0(X)$ be the set of points of $F_0$ of height at
most $X$, it then follows that
$$|V_{\Z}\cap F_0(X)|=cX^a\log^b X + O(X^{a-\delta}).$$ It remains to
estimate the number of degenerate points inside the main body
$F_0\subset F$, and it is in this last estimate that past results have
frequently failed to obtain a power saving.

The typical argument runs as follows: the reduction modulo a prime $p$
of $V_{\Z}^{\deg}$ is shown to lie in a subset $B_p\subset V_{\F_p}$ of
density $\mu_p$, which approaches a constant $c$ between $0$ and $1$
as $p\rightarrow\infty$. Set $\widetilde{B}_p$ to be the set of
elements of $V_{\Z}$ reducing to $B_p$. For any finite fixed set $S$
of primes, one has the estimate
$$|V^{\deg}_{\Z}\cap F_0(X)|\leq |\bigcap_{p\in S}\widetilde{B}_p\cap
F_0(X)|\sim \prod_{p\in S} \mu_p\cdot cX^a\log^b X$$ This is true for
every fixed $S$. Since $\prod_{p\in S}\mu_p$ can be made arbitrarily
small by picking $S$ to be a large set, one obtains
$$
|V^{\deg}_{\Z}\cap F_0(X)|=o(X^a\log^b X).
$$

However it is possible to do much better by estimating $|\bigcap_{p\in S}
\widetilde{B}_p|$ with the Selberg sieve \cite[Theorem 6.4]{IK}. To
apply this sieve, we need the following uniform statement: let $L\subset V_\Z$ be defined by congruence conditions modulo $m$. Then
$$
|L\cap F_0(X)|=\mu(L)cX^a\log^b X +O(X^{a-\delta}m^A),
$$ where $\mu(L)$ denotes the density of $L$ in $V_\Z$, and $A$ is a
fixed constant independent of $L$. The application of the Selberg
sieve immediately yields a power saving error term:
$$
|V^{\deg}_{\Z}\cap F_0(X)|=O_\epsilon(X^{a-\frac{\delta}{2A+3}+\epsilon}).
$$

We remark that for arithmetic applications one usually needs a further sieve (for
example, a sieve from quintic rings to maximal quintic rings). This
can be done with a power saving error term following \cite{BBP}.

\subsection{Outline of the paper}
In \S2, we collect the arguments used by Bhargava in \cite{dodpf} to
parametrize and count the number of quintic rings of a bounded discriminant. In \S3 we use the Selberg Sieve to obtain
a power saving estimate for the number of non $S_5$-orders - what we call $V_{\Z}^{\deg}$. We try to adhere to the notation
of \cite[Theorem 6.4]{IK} for the convenience to the reader. In section 4 we prove our main theorem by sieving down from $S_5$-orders to
$S_5$-fields.

\section{$S_5$-quintic orders}
In this section, we recall results from \cite{dodpf} that allow us to
obtain asymptotics for the number of $S_5$-quintic orders having
bounded discriminant. All the results and the notation in this section
directly follow \cite{dodpf}.

\subsection{Parametrizing quintic rings}
Let $V_\Z$ denote the space of quadruples of $5\times 5$-alternating
matrices with integer coefficients. The group
$G_\Z:=\GL_4(\Z)\times\SL_5(\Z)$ acts on $V_\Z$ via $(g_4,g_5)\cdot
(A,B,C,D)^t=g_4(g_5Ag_5^t,g_5Bg_5^t,g_5Cg_5^t,g_5Dg_5^t)^t$. The ring
of invariants for this action is generated by one element, denoted the
discriminant. In \cite{quintic}, Bhargava shows that quintic rings are
parametrized by $G_\Z$-orbits on $V_\Z$:

\begin{theorem}[Bhargava]
There is a canonical bijection between the set of $G_\Z$-orbits on
elements $(A,B,C,D,E)\in V_\Z$ and the set of isomorphism classes of
pairs $(R,R')$, where $R$ is a quintic ring and $R'$ is a sextic
resolvent of $R$. Under this bijection, we have
$\Disc(A,B,C,D,E)=\Disc(R)=\frac{1}{16}\Disc(R')^{1/3}$.
\end{theorem}

\subsection{Counting quintic rings}
Let $V_\Z^\ndeg$ denote the set of elements in
$V_\Z$ that correspond to orders in
$S_5$-fields, and let $V_\Z^\deg$ be $V_\Z\backslash V_\Z^\ndeg$.
For a $G_\Z$-invariant subset $S$ of $V_\Z$, let $N(S,X)$ denote the
number of irreducible $G_\Z$-orbits on $S^\ndeg:=S\cap V_\Z^\ndeg$
having discriminant bounded by $X$.\footnote{In \cite{dodpf},
  Bhargava defined $N(S,X)$ slightly differently to also count orders
  in non $S_5$-fields.}


The quantity $N(V_\Z;X)$ is estimated in the following way: the action
of $G_\R$ on $V_\R$ has three open orbits denoted $V_\R^{(0)}$,
$V_\R^{(1)}$, and $V_\R^{(2)}$. Let $\FF$ be a fundamental domain for
the action of $G_\Z$ on $G_\R$ and let $H$ be an open bounded set in
$V_\R^{(i)}$. Denote $V_\Z\cap V_\R$ by $V_\Z^{(i)}$, and let
$S\subset V_\Z^{(i)}$ be a $G_\Z$-invariant subset. Then by
\cite[Equations (9) and (10)]{dodpf}, we have
\begin{equation}\label{eq1}
\begin{array}{rcl}
N(S,X)&=&\displaystyle\frac{\displaystyle\int_{v\in H}\#\{x\in \FF
  v\cap
  S^\ndeg:|\Disc(x)<X\}|\Disc(v)|^{-1}dv}{n_i\displaystyle\int_{v\in
    H}|\Disc(v)|^{-1}dv}\\[0.2in]
&=&C_i\displaystyle\int_{g\in\FF}\#\{x\in g H\cap
S^\ndeg:|\Disc(x)<X|\}dg,
\end{array}
\end{equation}
where $dg$ is Haar-measure on $G_\R$, $n_i$ depends only on $i$, and
$C_i$ is independent of $S$.  In what follows, we pick $\FF$ and $dg$
as in \cite[Section 2.1]{dodpf}. Once picked, we let \eqref{eq1}
define $N(S,X)$ even for sets $S$ that are not $G_\Z$-invariant.
Define also the related quantity $N^*(S,X)$ via
$$ N^*(S,X):=C_i\displaystyle\int_{g\in\FF}\#\{x\in g H\cap
S:|\Disc(x)<X|\}dg.
$$
Letting $a_{12}$ denote the $12$-coordinate of $A$, \cite[Lemma 11]{dodpf} states that we have
$$
N(\{x\in V^{(i)}_\Z:a_{12}=0\},X)=O(X^{\frac{39}{40}}).
$$
Proposition 12 combined with the last equation in Section 2.6 of \cite{dodpf} imply that
\begin{equation}\label{cutoff}
 N^*(\{x\in V_\Z^{(i)}:a_{12}\neq
0\},X)=
c_iX+O(X^{\frac{39}{40}}),
\end{equation}
where $$c_i:=\frac{\zeta(2)^2\zeta(3)^2\zeta(4)^2\zeta(5)}{2n_i}.$$

To sieve down to fields, we will need analogous equations where
$V_\Z^{(i)}$ is replaced by a set defined by finitely many congruence
conditions on $V_Z$. Specifically, if $L$ is a translate of $mV_\Z$, then from \cite[Equation 28]{dodpf} we have
\begin{equation}\label{eq2}
N^*(\{x\in L\cap V_\Z^{(i)}:a_{12}\neq 0\},X)=c_im^{-40}X+O(m^{-39}X^{\frac{39}{40}}).
\end{equation}

\subsection{Congruence conditions for $V_\Z^\deg$}\label{badcong}
As explained in \cite[Section 3.2]{dodpf}, there exist disjoint
subsets $T_p(1112)$ and $T_p(5)$ of $V_\Z$, that are defined by
congruence conditions modulo $p$, such that for any two distinct primes $p$ and
$q$, the set $V_{\Z}^{\deg}$ is disjoint from $T_p(1112)\cap
T_q(5)$. Furthermore, the densities $g_p(1112)$ of $T_p(1112)$ and
$g_p(5)$ of $T_p(5)$ approach $1/12$ and $1/5$, respectively as
$p\to\infty$.  We set $S_p(1112)$ and $S_p(5)$ be the complements of
$T_p(1112)$ and $T_p(5)$ respectively.

\section{Applying the Selberg Sieve}

In this section we give a power saving estimate for
$N^*(V^{\deg,(i)}_{\Z},X)$. By section \ref{badcong}, we know that
\begin{equation}\label{upest}
N^*(V^{\deg,(i)}_{\Z},X)\leq N^*(\cap_p S_p(5),X)  + N^*(\cap_p S_p(1112),X).
\end{equation}

Our goal is to bound each of the two terms on the RHS of \eqref{upest}
using the Selberg Sieve. We turn to the details.  We being by fixing a
number $z<X$. Set $P(z)=\prod_{p<z} p.$ For each square-free number
$d|P(z)$, set $g_d(5)=\prod_{p\mid d}g_p(5)$
and $$a_d(5)=N^*\left(\bigcap_{p\mid d}T_p(5)\bigcap_{p\mid
  \frac{P(z)}{d}}S_p(5),X\right).$$ This is a sequence of non-negative
integers, and by \eqref{eq2} we have that for all $d\mid P(z)$,
\begin{equation}\label{congsieve}
\sum_{n\equiv 0\textrm{ mod } d} a_n(5) = N^*(\cap_{p\mid d}T_p(5),X) = c_ig_d(5)X+r_d
\end{equation}
where $r_d=O(dg_d(5)X^{\frac{39}{40}})$.
Fix $D>1$ and
define
$$h_d(5)=\prod_{p\mid d}\frac{g_p(5)}{1-g_p(5)},\quad\quad H=\sum_{\substack{d<\sqrt{D}\\d\mid P(z)}}h_d(5).$$
A direct application of \cite[Theorem
  6.4]{IK} yields
\begin{equation}\label{mainest}
a_1(5)= \sum_{(n,P(z))=1}a_n(5)\leq
C_iXH^{-1} +
O\left(\sum_{d<D,d\mid P(z)} \tau_3(d) r_d\right).
\end{equation}

To use \eqref{mainest} we take $z\rightarrow\infty$. Note that since
$g_p(5)\rightarrow\frac{1}{5}$, we have
$$d^{-\epsilon}\ll_\e
g_d(5),h_d(5)\ll_\e d^{\epsilon}.$$ It follows that $H = D^{\frac12+o(1)}$
while
$$\left|\sum_{d<D,d\mid P(z)} \tau_3(d) r_d\right|\ll_\e
X^{\frac{39}{40}}D^{\epsilon}\sum_{d<D}d\leq X^{\frac{39}{40}}D^{2+\epsilon}.$$

We deduce that $a_1(5)\ll_\epsilon XD^{-\frac12+\epsilon} +
X^{\frac{39}{40}}D^{2+\epsilon}$. Optimizing, we take $D=X^{\frac{1}{100}}$ to deduce
that $a_1(5)\ll_{\epsilon} X^{\frac{199}{200}+\epsilon}$.

It follows that $$N^*(\cap_p S_p(5),X)\leq N^*(\cap_{p<z}
S_p(5),X)=a_1(5)\ll_{\epsilon} X^{\frac{199}{200}+\epsilon}.$$ The case of
$N^*(\cap_p S_p(1112),X)$ can be treated similarly, and we thus
conclude by \eqref{upest} that
\begin{equation}\label{finest}
N^*(V^{\deg,(i)}_{\Z},X)\ll_{\epsilon} X^{\frac{199}{200}+\epsilon}.
\end{equation}
\section{Sieving to fields}
In this section we follow \cite{BBP} to prove Theorem \ref{main}.  For $d$ square-free,
define $W_d\subset V_Z$ to be the set of elements corresponding to
quintic orders that are not maximal at each prime dividing $d$. Recall from
\cite{dodpf} that $W_d$ is defined by congruence conditions modulo $d^2$.

We need a slight generalization of the uniformity estimate
\cite[proposition 19]{dodpf}.

\begin{lemma}\label{uniform}
$N(W_d,X) = O_{\epsilon}(X/d^{2-\epsilon})$
\end{lemma}

\begin{proof}
As in \cite[proposition 19]{dodpf} we count rings that are not maximal
by counting their over-rings. As in that proof, we use the result of
Brakenhoff \cite{B} that the number of orders having index $m$ in a
maximal quintic ring $R$ is $\prod_{p^k\mid\mid m}
O(p^{\min(2k-2,\frac{20k}{11})}).$ Moreover, the number of sextic
resolvents of a quintic ring of content $n$ is $O(n^6)$. (Recall that
the content of a ring is the largest integer $n$ such that $R=\Z+nR'$
for some quintic ring $R'$.)

Thus we have that $$N(W_d,X)\ll_\e d^\e X\sum_{n=1}^\infty \frac{n^6}{n^8}
\prod_{p\mid d}\sum_{k=1}^{\infty}
\frac{p^{\min(2k-2,\frac{20k}{11})}}{p^{2k}}
\ll_{\e}X/d^{2-\e}$$ as desired.
\end{proof}

Define $$N^*_{12}   (S,X) = N^*(\{x\in S: a_{12}\neq 0\}, X).$$
Now, a point in $V_{\Z}$ corresponds to a maximal order in an $S_5$-field precisely if it is in
$\cap_p W_p\cap V_{\Z}^{\ndeg}$.  Denote the density of $W_d$ by $c_d$, and recall from \cite{dodpf} that $c_d=O_{\e}(d^{-2+\e})$. We have

\begin{align*}
N(\cap_p W_p\cap V_{\Z}^{(i)}, X) &= \sum_{d\in\N} \mu(d)N(W_d,X)\\ &=
\sum_{d<T} \left(c_i\mu(d)c_d X + O(X^{\frac{39}{40}}d^2)-\mu(d)N^*_{12}(W_d\cap
V_{\Z}^{\deg,(i)},X)\right) + \sum_{d>T} O_{\e}(X/d^{2-\e})\\ &=
\sum_{d<T} \left(c_i\mu(d)c_d X - \mu(d)N^*_{12}(W_d\cap
V_{\Z}^{\deg,(i)},X)\right)+O(X/T^{1-\e}+X^{\frac{39}{40}}T^{3+\e})\\ &= \sum_{d\in\N}
c_i\mu(d)c_d X
+O_{\e}(X/T^{1-\e}+X^{\frac{39}{40}}T^{3+\e}+ X^{\frac{199}{200}}T^{1+\e})\\ &= c_i\prod_p (1-c_p)X + O_{\e}(X/T^{1-\e}+X^{\frac{39}{40}}T^{3+\e}+X^{\frac{199}{200}}T^{1+\e}), \\
\end{align*}
where the second equality follows from \eqref{cutoff} and Lemma \ref{uniform}, and the fourth equality follows from \eqref{finest}.
Optimizing, we pick $T=X^{\frac{1}{400}}$ to obtain Theorem \ref{main}.


\begin{thebibliography}{10}
\bibitem{Bhyper} M.\ Bhargava, Most hyperelliptic curves over Q have
  no rational points, arxiv/1308.0395.
\end{thebibliography}
\end{document}